\documentclass[12pt,a4paper]{amsart}
\usepackage{dsfont}
\usepackage{graphicx}
\usepackage{amssymb}
\usepackage[all]{xy}
\usepackage{amsmath}
\usepackage{tikz}
\usepackage[french,english]{babel}
\usepackage{amssymb,amsmath,amsfonts,amsthm,amscd}
\usepackage{indentfirst,graphicx}
\usepackage[font={scriptsize},captionskip=5pt]{subfig}
\usepackage{tikz}
\usetikzlibrary{matrix}
\usepackage[T1]{fontenc}
\usepackage[utf8]{inputenc}

\newtheorem{theorem}{Theorem}[section]
\newtheorem{corollary}[theorem]{Corollary}
\newtheorem{proposition}[theorem]{Proposition}
\newtheorem{lemma}[theorem]{Lemma}

\newtheorem{definition}[theorem]{Definition}

\newtheorem{example}[theorem]{Example}
\newtheorem{remark}[theorem]{Remark}

\newcommand{\cB}{{\mathcal B}}

\newcommand{\cJ}{{\mathcal J}}

\newcommand{\cL}{{\mathcal L}}
\newcommand{\cO}{{\mathcal O}}

\newcommand{\cX}{{\mathcal X}}

\newcommand{\cM}{{\mathcal M}}

\newcommand{\cOXp}{{\cO_X^p}}
\newcommand{\bC}{{\mathbb{C}}}
\newcommand{\bP}{{\mathbb{P}}}
\newcommand{\bN}{{\mathbb{N}}}

\newcommand{\norm}{{\parallel}}

\begin{document}

\title{THE LIPSCHITZ SATURATION OF A MODULE}
\author{Terence Gaffney and Thiago F. da Silva}

\maketitle

\begin{abstract}
{\small In this work we extend the concept of the Lipschitz saturation of an ideal defined in \cite{G2} to the context of modules in some different ways, and we prove they are generically equivalent.}
\end{abstract}

\let\thefootnote\relax\footnote{2010 \textit{Mathematics Subjects Classification} 32S15, 14J17, 32S60
	
	}

\section*{Introduction}

The definition of the Lipschitz saturation of an ideal appears in \cite{G2}, in the context of bi-Lipschitz equisingularity. The study of bi-Lipschitz equisingularity was started by Zariski \cite{Z}, Pham and Teissier \cite{PT,P1}, and was further developed by Lipman \cite{L}, Mostowski \cite{M1,M2}, Parusinski \cite{PA1,PA2}, Birbrair \cite{B} and others.

The Lipschitz Saturation and the double of an ideal $I$, denoted $I_S$ and $I_D$, respectively, were defined in \cite{G2}, where $I$ is a sheaf of ideals of $\cO_X$, the analytic local ring of an analytic variety $X$. The ideal $I_S$ consists of elements in $\cO_X$ whose quotient of its pullback by the blowup-map, with a local generator of the pullback of $I$ is Lipschitz. The double $I_D$ is the submodule of $\cO_{X\times X}^2$ generated by $(h\circ\pi_1,h\circ\pi_2)$, $h\in I$, where $\pi_1,\pi_2:X\times X\rightarrow X$ are the projections. Theorem 2.3 of \cite{G2} gives an equivalence between $I_S$ and the integral closure of $I_D$.

In \cite{SG} the authors generalize the notion of the double for a sheaf of modules $\cM$, and the Genericity Theorem is generalized for an arbitrary family of analytic varieties. They also get an interesting decomposition of the stalk of $\cM_D$ at $(x,x')$, as the direct sum of the stalks $\cM_{x}$ and $\cM_{x'}$, as long $x\neq x'$. If $\cM$ is the jacobian module of a family of analytic varieties, the stalks at $x$ and $x'$ determine the tangent hyperplanes at these two points. Since, to control the Lipschitz behavior of the tangent hyperplanes to $X$, it is natural to look for a sheaf on $X\times X$ whose stalks determine the tangent hyperplanes at each pair of distinct points, it is natural to consider the double of the jacobian module. 

This paper is organized as follows:

In section \ref{sec0} we recall some basic background material.

In section \ref{sec3} we present some definitions for the Lipschitz Saturation of a module, inspired by Theorem 2.3 of \cite{G2}, and also motivated by characterizations of the integral closure of modules coming from the integral closure of ideals. Despite all these definitions of Lipschitz saturation agree on the ideal case, we show these notions are not the same anymore in the module case. Nevertheless, we prove all the definitions of Lipschitz saturation of modules are generically equivalent. 

In section \ref{sec4} we prove the Necessity theorem, and we relate the infinitesimal Lipschitz condition A coming from the first Lipschitz saturation with the strongly bi-Lipschitz equisingularity and the Mostowski's conditions for a Lipschitz stratification. A new proof of the Genericity Theorem arises using the Necessity Theorem and the Mostowski's results about Lipschitz stratifications.

\section*{Acknowledgements}

The authors are grateful to Nivaldo Grulha Jr. for his careful reading of this work and to Maria Aparecida Soares Ruas for the valuable conversations about the subject of this work, especially in the generic equivalence among the Lipschitz saturations defined here.

The first author was supported in part by PVE-CNPq, grant 401565/2014-9. The second author was supported by Funda\c{c}\~ao de Amparo \`a Pesquisa do Estado de S\~ao Paulo - FAPESP, Brazil, grant 2013/22411-2.

\section{Background on Lipschitz Saturation of Ideals and Integral Closure of Modules}\label{sec0}

The Lipschitz saturation of a local ring was defined by Pham and  Teissier in \cite{PT}.

\begin{definition}
Let $I$ be an ideal of $\cO_{X,x}$,  $SB_I(X)$ the saturation of the blow-up and $\pi_S: SB_I(X)\rightarrow X$ the projection map. The {\bf Lipschitz saturation} of the ideal $I$ is denoted $I_S$, and is the ideal $I_S:=\{h\in\cO_{X,x}\mid\pi_S^*(h)\in\pi_S^*(I)\}$.
\end{definition}

Since the normalization of a local ring $A$ contains the Lipschitz Saturation of $A$, it follows that $I\subseteq I_S\subseteq\overline{I}$. In particular, if $I$ is integrally closed then $I_S=\overline{I}$.

This definition can be given an equivalent statement using the theory of integral closure of modules. Since Lipschitz conditions depend on controlling functions at two different points as the points come together, we should look for a sheaf defined on $X\times X$. We describe a way of moving from a sheaf of ideals on $X$ to a sheaf on $X\times X$. 

Let $\pi_1,\pi_2: X\times X\rightarrow X$ be the projections to the i-th factor, and let $h\in\cO_{X,x}$. Define $h_D\in\cO_{X\times X,(x,x)}^{2}$ as $(h\circ\pi_1,h\circ\pi_2)$, called the double of $h$. We define the double of the ideal $I$, denoted $I_D$, as the submodule of $\cO_{X\times X,(x,x)}^2$ generated by $h_D$, where $h$ is an element of $I$.

We can see in \cite{G2}, the following result gives a link between Lipschitz saturation and integral closure of modules.

\begin{theorem}[\cite{G2}, Theorem 2.3]
Suppose $(X,x)$ is a complex analytic set germ, $I\subseteq\cO_{X,x}$ and $h\in\cO_{X,x}$.  
\noindent Then $h\in I_S$ if, and only if, $h_D\in\overline{I_D}$. 
\end{theorem}

Using the Lipschitz saturation of ideals (and doubles), in \cite{G1}, we defined the infinitesimal Lipschitz conditions for hypersurfaces.

Let us recall two results about the integral closure of modules which will inspire good definitions for Lipschitz saturation of modules.

\vspace{0,5cm}

{\it The ideal sheaf $\rho(\cM)$ on $X\times \bP^{p-1}$ associated to a submodule sheaf $\cM$ of $\cOXp$ (see \cite{GK})}: Given $h=(h_1,...,h_p)\in\cOXp$ and $(x,[T_1,...,T_p])\in X\times\bP^{p-1}$, with $T_i\neq0$, we define $\rho(h)$ as the germ of the analytic map given by $\sum\limits_{j=1}^p h_j(z)\frac{T_j}{T_i}$ which is well-defined on a Zariski open subset of $X\times\bP^{p-1}$ that contains the point $(x,[T_1,...,T_p])$. We define $\rho(\cM)$ as the ideal generated by $\{\rho(h)\mid h\in \cM\}$. The next result, proved in \cite{GK}, gives a strong relation between the integral closure of modules and ideals.

\begin{proposition}[\cite{GK}, Proposition 3.4]\label{proposition GK}
Let $h\in\cO_{X,x}^p$. Then $h\in\overline{\cM}$ at $x$ if, and only if, $\rho(h)\in\overline{\rho(\cM)}$ at all point $(x,[T_1,...,T_p])\in V(\rho(\cM))$.
\end{proposition}

In \cite{G3} there is another way to make a link between the integral closure of modules and ideals, using minors of a matrix of generators of $\cM$. 

Let $\cM$ be a submodule sheaf of $\cOXp$, and $[\cM]$ a matrix of generators of $\cM$. For each $k$, let $J_k(\cM)$ denote the ideal of $\cO_X$ generated by the $k\times k$ minors of $[\cM]$. If $h\in\cOXp$, let $(h,\cM)$ be the submodule generated by $h$ and $\cM$.

\begin{proposition}[\cite{G3}, Corollary 1.8]\label{proposition G}
Suppose $(X,x)$ is a complex analytic germ with irreducible components $\{V_i\}$. Then, $h\in\overline{\cM}$ at $x$ if, and only if, $J_{k_i}((h,\cM_i))\subseteq \overline {J_{k_i}(\cM_i)}$ at $x$, where $\cM_i$ is the submodule of $\cO_{V_i,x}^p$ induced from $\cM$ and $k_i$ is the generic rank of $(h,\cM_i)$ on $V_i$. 
\end{proposition}

Now we recall some basic results about the double of a module developed in \cite{SG}.

Let $X \subseteq \mathbb{C}^n$ be an analytic space, and let $\cM$ be an $\cO_X $-submodule of $\cO_X^p$. Consider the projection maps $\pi_1,\pi_2: X \times X \rightarrow X$. We assume that $\cM$ is finitely generated by global sections.

\begin{definition}
	Let $h\in\cOXp$. The double of $h$ is defined as the element
	
	\noindent $h_D:=(h\circ\pi_1,h\circ\pi_2) \in\cO_{X\times X}^{2p}$.
	
	The double of $\cM$ is denoted by $\cM_D$, and is defined as the $\cO_{X\times X}$-submodule of $\cO_{X\times X}^{2p}$ generated by $\{h_D \mid  h\in \cM\}$.
	
\end{definition}

Consider $z_1,...,z_n$ the coordinates on $\mathbb{C}^n$. The next lemma is a useful tool to deal with the double.

\begin{lemma}[\cite{SG}, Lemma 2.2]\label{L1} Suppose $\alpha\in\cO_X$ and $h\in\cOXp$. Then:

	\begin{enumerate}
		\item $(\alpha h)_D=-(0,(\alpha\circ\pi_1-\alpha\circ\pi_2)(h\circ\pi_2))+(\alpha\circ\pi_1)h_D$;
		
		\item $(0,(\alpha\circ\pi_1-\alpha\circ\pi_2)(h\circ\pi_2))\in \cM_D$, for all $h\in \cM$ and $\alpha\in\cO_X$;
		
		\item $\alpha\circ\pi_1-\alpha\circ\pi_2 \in I(\Delta(X))=(z_1\circ\pi_1-z_1\circ\pi_2,\ldots, z_n\circ\pi_1-z_n\circ\pi_2)$, for all $\alpha\in\cO_X$;
		
		\item $(g+h)_D=g_D+h_D$, for all $g,h\in\cOXp$. 
	\end{enumerate}
\end{lemma}

In \cite{SG} we see that it is possible to obtain a set of generators for $\cM_D$ from a set of generators of $\cM$.

\begin{proposition}[\cite{SG}, Proposition 2.3]\label{P2}
	Suppose that $\cM$ is generated by $\{h_1,\ldots,h_r \}$. Then, the following sets are generators of $\cM_D$:
	\begin{enumerate}
		
		\item $\cB=\{(h_1)_D,\ldots,(h_r)_D\} \cup \{(0_{\cO_{X\times X}^{p}},(z_i\circ\pi_1-z_i\circ\pi_2)(h_j\circ\pi_2))$ $|$ $i\in\{1,\ldots,n\}$ and $j\in\{1,\ldots,r\}\}$.
		
		\item $\cB'=\{(h_1)_D,\ldots,(h_r)_D\} \cup \{((z_i\circ\pi_1-z_i\circ\pi_2)(h_j\circ\pi_1),0_{\cO_{X\times X}^{p}}))$ $|$ $i\in\{1,\ldots,n\}$ and $j\in\{1,\ldots,r\}\}$.
		
		\item $\cB''=\{(h_1)_D,\ldots,(h_r)_D\} \cup \{(z_ih_j)_D$ $|$ $i\in\{1,\ldots,n\}$ and $j\in\{1,\ldots,r\}\}$.
		
	\end{enumerate}
	
\end{proposition}

The following proposition gives a link between the integral closure of the double of a module and the integral closure of the original module.

\begin{proposition}[\cite{SG}, Proposition 2.9]\label{P12}
	Let $h\in\cOXp$. 
	\begin{enumerate}
		\item If $h_D\in\overline{\cM_D}$ at $(x,x')$ then $h\in\overline{\cM}$ at $x$ and $x'$. 
		\item If $h_D\in (\cM_D)^{\dagger}$ at $(x,x')$ then $h\in \cM^{\dagger}$ at $x$ and $x'$.
	\end{enumerate}
\end{proposition}

In the next theorem we compute the generic rank of the double of a module.

\begin{theorem}[\cite{SG}, Proposition 2.5]\label{T2.9}
	Let $(X,x)$ be an irreducible analytic complex germ of dimension $d\ge 1$, and $\cM\subseteq\cO_{X,x}^p$ a submodule of generic rank $k$. Then $\cM_D$ has generic rank $2k$ at $(x,x)$. 
\end{theorem}

\begin{corollary}[\cite{SG}, Corollary 2.6]\label{C2.10}
	Let $\{V_i\}$ be the irreducible components of $(X,x)$. For each $i$, if $\cM$ has generic rank $k_i$ on $V_i$ then $\cM_D$ has generic rank $2k_i$ on $V_i\times V_i$. In particular, if $\cM$ has generic rank $k$ on each component of $X$ then $\cM_D$ has generic rank $2k$ on each component of $X\times X$.
\end{corollary} 

We end this section recalling an important result which states that the double of a sheaf of modules $\cM$ carries all the information at $(x,x')$ as the stalks of $\cM$ do at $x$ and $x'$, as long $x\neq x'$.

\begin{proposition}[\cite{SG}, Proposition 2.11]\label{P2.13}
	Let $\cM\subseteq\cO_X^p$ be a sheaf of submodules. Consider 
	$(x,x')\in X\times X$ with $x\neq x'$. Then:
	$$\cM_D=(\cM_x\circ\pi_1)\oplus (\cM_{x'}\circ\pi_2)$$ at $(x,x')$.
\end{proposition}

Proposition \ref{P2.13} provides additional motivation for the idea of the double: In order to control the Lipschitz behavior of pairs of tangent planes at two different points $x$ and $x'$ of a family $\cX$, it is helpful to have each module which determines the tangent hyperplanes at each point as part of the construction. Furthermore, this proposition shows that $JM(\cX)_D$ at $(x,x')$ contains both $JM(\cX)_x$ and $JM(\cX)_{x'}$.

\section{The Lipschitz Saturation of a Module}\label{sec3}

We want extend to modules the notion of Lipchitz saturation that was defined in \cite{G2} for ideals. Motivated by some equivalent descriptions that the Lipschitz saturation has in the ideal case, we define corresponding versions of the Lipschitz saturation for modules, and explore the extent to which they are equivalent.

The main motivation of the next definition is Theorem 2.3 of \cite{G2}.

Let $X\subseteq \bC^n$ be an analytic set and $\cM$ be an $\cO_X$-submodule of $\cOXp$.

\begin{definition}
The {\bf 1-Lipschitz saturation of $\cM$ at $x\in X$} is denoted by $(\cM_{S_1})_x$, and is defined by $$(\cM_{S_1})_x:=\{h\in\cO_{X,x}^{p}\mid h_D\in\overline{\cM_D}\mbox{ at }(x,x)\}.$$

In the family case, the {\bf 1-Lipschitz saturation of $\cM$ at $x\in X$ relative to $Y$} is defined as above taking the double relative to $Y$. 
\end{definition}

\begin{proposition}\label{P15} Let $\cM$ be a sheaf of $\cO_X$submodules of $\cOXp$.
\begin{enumerate}
\item $\cM_{S_1}$ is an $\cO_X$-submodule of $\cOXp$;
\item $\cM\subseteq \cM_{S_1} \subseteq \overline{\cM}$. In particular, $\cM$ is a reduction of $\cM_{S_1}$ and $e(\cM,\cM_{S_1})=0$.
\end{enumerate}
The same result holds in the family case.
\end{proposition}

\begin{proof}
Let $x\in X$ be an arbitrary point.

(1) Let $\alpha\in\cO_{X,x}$ and $h,h'\in \cM_{S_1}$ at $x$. Since $h_D\in\overline{\cM_D}$ then by Proposition \ref{P12} (1) we have that $h\in\overline{\cM}$ at $x$. Thus, $(0,(\alpha\circ\pi_1-\alpha\circ\pi_2)(h\circ\pi_2))\in\overline{\cM_D}$. Hence
$(\alpha h+h')_D=(\alpha\circ\pi_1)h_D+(0,(\alpha\circ\pi_1-\alpha\circ\pi_2)(h\circ\pi_2))+h'_D\in\overline{\cM_D}$.

(2) If $h\in \cM$ at $x$ then $h_D\in \cM_D\subseteq\overline{\cM_D}$ at $(x,x)$, so $h\in \cM_{S_1}$ at $x$. Therefore, $\cM\subseteq \cM_{S_1}$.

Furthermore, if $h\in \cM_{S_1}$ at $x$ then $h_D\in\overline{\cM_D}$ at $(x,x)$, and by Proposition \ref{P12} (1) we have that $h\in\overline{\cM}$ at $x$. Therefore, $\cM_{S_1}\subseteq\overline{\cM}$.
\end{proof}

In order to define the second Lipschitz saturation, let us fix some notations.

For each $\psi:X\rightarrow \mbox{Hom}(\bC^p,\bC)$, $\psi=(\psi_1,...,\psi_p)$ and $h=(h_1,...,h_p)\in\cOXp$, we define $\psi\cdot h\in\cO_X$ given by $(\psi\cdot h)(z):=\sum\limits_{i=1}^{p}\psi_i(z)h_i(z)$. We define $\psi\cdot \cM$ as the ideal of $\cO_X$ generated $\{\psi\cdot h \mid h\in \cM\}$.

\begin{lemma}\label{L16}
\begin{enumerate}
\item $\psi\cdot(\alpha g+h)=\alpha(\psi\cdot g)+(\psi\cdot h)$, $\forall g,h\in \cOXp$ and $\alpha\in\cO_X$.
\item If $\cM$ is generated by $\{h_1,...,h_r\}$ then $\psi\cdot \cM$ is generated by $\{\psi\cdot h_1,...,\psi\cdot h_r\}$.
\end{enumerate}
\end{lemma}
\begin{proof}
It is easy to see that (2) is a straightforward consequence of (1). Now, write $g=(g_1,...,g_p)$ and $h=(h_1,...,h_p)$. Then, for every $z$ we have $(\psi\cdot(\alpha g+h))(z)=\sum\limits_{i=1} ^{p}\psi_i(z)(\alpha(z)g_i(z)+h_i(z))=\alpha(z)\sum\limits_{i=1}^{p}\psi_i(z)g_i(z)+\sum\limits_{i=1}^{p}\psi_i(z)h_i(z)=(\alpha(\psi\cdot g)+(\psi\cdot h))(z)$.
\end{proof}

Let us fix some notation for the minors of a matrix. Let $k\in\bN$ and let $A$ be a matrix. If $I=(i_1,...,i_k)$ and $J=(j_1,...,j_k)$ are $k$-indexes, $A_{IJ}$ is defined as the $k\times k$ submatrix of $A$ formed by the rows $i_1,...,i_k$ and  columns $j_1,...,j_k$ of $A$. We denote $J_{IJ}(A):=\det(A_{IJ})$.

\begin{lemma}\label{L21}
Suppose that $\cM$ has generic rank $k$ in each component of $X$. If $I=(i_1,...,i_k)$ and $J=(j_1,...,j_k)$ are indexes with $j_1=1$ then there exists $\psi: X\rightarrow \mbox{Hom}(\bC^p,\bC)$ such that:
\begin{enumerate}
\item $\psi\cdot h=J_{IJ}(h,\cM)$, $\forall h\in\cO_{X,x}^p$;
\item $\psi\cdot \cM\subseteq J_k(\cM)$.
\end{enumerate}
\end{lemma}

\begin{proof}
Let us fix a matrix of generators of $\cM$, $[\cM]=$
$
\begin{bmatrix}
   |   &           &   | \\
g_1 & \ldots & g_r \\
   |   &          &    |
\end{bmatrix}$. We have $J_{IJ}(h,\cM)=\det [h,\cM]_{IJ}$, where $$[h,\cM]_{IJ}=\begin{bmatrix}
h_{i_1}      &     g_{i_1,j_2-1}  &   \ldots    &     g_{i_1,j_k-1} \\
   \vdots     &         \vdots        &                &         \vdots        \\
h_{i_k}      &     g_{i_k,j_2-1}  &   \ldots    &     g_{i_k,j_k-1}
\end{bmatrix}$$ for all $h=(h_1,...,h_p)$. Let $G_{i_l,j_s-1}$ be the $(l,s)$-cofactor of $[h,\cM]_{IJ}$, for all $l,s\in\{1,...,k\}$. Notice that the $(l,1)$-cofactors $G_{i_l,0}$ do not depend of $h$. Then, $J_{IJ}(h,\cM)=\det[h,\cM]_{IJ}=\sum\limits_{l=1}^{k}G_{i_l,0}\cdot h_{i_l}$, for all $h$. Take $\psi: X\rightarrow \mbox{Hom}(\bC^p,\bC)$ given by $(\psi_1,...,\psi_p)$ where $\psi_{i_l}=G_{i_l,0}$, for all $l\in\{1,...k\}$, and $\psi_j=0$, for every index $j$ off $I$. Thus, for all $h\in\cO_{X,x}^p$ we get $J_{IJ}(h,\cM)=\psi_I\cdot h_I=\psi\cdot h$.

Now, let $g\in \cM$ arbitrary. Then $(g,\cM)=\cM$. By (1) we have $\psi\cdot g=J_{I,J}(g,\cM)\in J_k((g,\cM))=J_k(\cM)$, hence $\psi\cdot \cM\subseteq J_k(\cM)$.
\end{proof}

{\bf Remark:} It is enough work with indexes $I=(i_1,...,i_k)$ and $J=(j_1,...,j_k)$ such that $i_1<...<i_k$ and $j_1<...<j_k$. 

Propositions \ref{proposition GK} and \ref{proposition G} are characterizations for the integral closure of modules using integral closure of ideals. Notice the next proposition is another characterization of the same type.

\begin{proposition}\label{P18}
Let $h\in\cO_{X,x}^p$ and suppose $\cM$ has generic rank $k$ on each component of $X$. Then, $h\in\overline{\cM}$ at $x$ if, and only if, $\psi\cdot h\in\overline{\psi\cdot \cM}$, at $x$, $\forall \psi:X\rightarrow \mbox{Hom}(\bC^p,\bC)$.
\end{proposition}

\begin{proof}

$(\Longrightarrow)$ It is a straightforward consequence of the curve criterion.

$(\Longleftarrow)$ The proof now use 1.7 and 1.8 of \cite{G3}.

By these results it is enough to check that $J_{IJ}(h,\cM)\in\overline{J_k(\cM)}$, for all indexes $I$ and $J$. Write $I=(i_1,...,i_k)$ and $J=(j_1,...,j_k)$. Let $[\cM]=$
$
\begin{bmatrix}
   |   &           &   | \\
g_1 & \ldots & g_r \\
   |   &          &    |
\end{bmatrix}$ be a matrix of generators of $\cM$. Write $h=(h_1,...,h_p)$. Then, $[h,\cM]=$
$
\begin{bmatrix}
   |    &  |   &           &   | \\
   h   & g_1 & \ldots & g_r \\
    |   &  |   &          &    |
\end{bmatrix}$. If $j_1>1$ then $J_{IJ}(h,\cM)$ is a $k\times k$ minor taken only among the generators of $\cM$, hence $J_{IJ}(h,\cM)\in J_k(\cM)\subseteq\overline{J_k(\cM)}$. 

Now suppose that $j_1=1$. By Lemma \ref{L21} there exists $\psi:X\rightarrow \mbox{Hom}(\bC^p,\bC)$ such that $\psi\cdot h=J_{I,J}(h,\cM)$ and $\psi\cdot \cM\subseteq J_k(\cM)$. By the hypothesis we have $\psi\cdot h\in\overline{\psi\cdot \cM}$, hence $J_{I,J}(h,\cM)=\psi\cdot h\in\overline{\psi\cdot \cM}\subseteq\overline{J_k(\cM)}$.
\end{proof}

Given $h\in\cO_{X,x}^p$ and $\cM\subseteq\cO_{X,x}^p$ submodule, we can ask if $h\in\overline{\cM}$ or $\rho(h)\in\overline{\rho(\cM)}$. For integral closure these questions have the same answer. Then we can ask the analogous questions for the Lipschitz saturation of $\cM$. The next version of a Lipschitz saturation is motivated by working with $\rho(h)$ and $\rho(\cM)$.

\begin{definition}
The {\bf 2-Lipschitz saturation of $\cM$ at $x\in X$} is denoted by $(\cM_{S_2})_x$, and is defined by $$(\cM_{S_2})_x:=\{h\in\cO_{X,x}^{p}\mid\psi\cdot h\in(\psi\cdot \cM)_S\mbox{ at }x\mbox{, }\forall \psi:X\rightarrow \textrm{Hom}(\bC^p,\bC)\}.$$

In the family case, the {\bf 2-Lipschitz saturation of $\cM$ at $x\in X$ relative to $Y$} is defined as above taking the Lipschitz saturation of $\psi\cdot \cM$ relative to $Y$.
\end{definition}

\begin{proposition}\label{P19} Let $\cM$ be a sheaf of $\cO_X$submodules of $\cOXp$.
\begin{enumerate}
\item $\cM_{S_2}$ is an $\cO_X$-submodule of $\cOXp$;
\item $\cM\subseteq \cM_{S_2} \subseteq \overline{\cM}$. In particular, $\cM$ is a reduction of $\cM_{S_2}$ and  $e(\cM,\cM_{S_2})=0$.
\end{enumerate}
The same result holds in the family case.
\end{proposition}

\begin{proof}
Let $x\in X$ be an arbitrary point.

(1) Let $h,h'\in \cM_{S_2}$ and $\alpha\in\cO_X$ at $x$, and let $\psi:X\rightarrow \mbox{Hom}(\bC^p,\bC)\}$ arbitrary. Then, $\psi\cdot h,\psi\cdot h'\in(\psi\cdot \cM)_S$ at $x$, and by Lemma \ref{L16} (1) we have $\psi\cdot(\alpha h+h')=\alpha(\psi\cdot h)+\psi\cdot h'\in(\psi\cdot \cM)_S$ at $x$. Therefore, $\alpha h+h'\in \cM_{S_2}$ at $x$.

(2) If $h\in \cM$ then $\psi\cdot h\in\psi\cdot \cM\subseteq(\psi\cdot \cM)_S$, so $h\in \cM_{S_2}$ and $\cM\subseteq \cM_{S_2}$.

Now, let $h\in \cM_{S_2}$. Then, $\psi\cdot h\in(\psi\cdot \cM)_S\subseteq\overline{\psi\cdot \cM}$, $\forall\psi:X\rightarrow \mbox{Hom}(\bC^p,\bC)$. By Proposition \ref{P18} we conclude that $h\in\overline{\cM}$.
\end{proof}

Next, we begin to compare the first and second  Lipschitz Saturations.

\begin{proposition}\label{P23}
$\cM_{S_1}\subseteq \cM_{S_2}$ at every point $x\in X$.
\end{proposition}

\begin{proof}
Let $h=(h_1,...,h_p)\in \cM_{S_1}$ at $x$. Then $h_D\in \overline{\cM_D}$ at $(x,x)$. We need to check that $(\psi\cdot h)_D\in\overline{(\psi\cdot \cM)_D}$, for all $\psi:X\rightarrow \mbox{Hom}(\bC^p,\bC)$. Let $\phi=(\phi_1,\phi_2):(\bC,0)\rightarrow(X\times X,(x,x))$ be an arbitrary analytic curve. Since $h_D\in\overline{\cM_D}$ then we can write $$h_D\circ\phi=\sum\limits_{j}\alpha_j\phi^*((g_j)_D)$$ with $g_j=(g_{1j},...,g_{pj})\in \cM$ and $\alpha_j\in\cO_{\bC,0}$ for all $j$. Looking to the above equation and comparing the $2p$ coordinates we conclude that $h_i\circ\phi_1=\sum\limits_{j}\alpha_j(g_{ij}\circ\phi_1)$ and $h_i\circ\phi_2=\sum\limits_{j}\alpha_j(g_{ij}\circ\phi_2)$, for all $i\in\{1,...,p\}$. Let $\psi_1,...,\psi_p$ be the coordinate functions of $\psi$.

Thus: $(\psi\cdot h)_D\circ\phi=(\sum\limits_{i}(\psi_i\circ\phi_1)\cdot(h_i\circ\phi_1),\sum\limits_{i}(\psi_i\circ\phi_2)\cdot(h_i\circ\phi_2))\\=(\sum\limits_{i,j}(\psi_i\circ\phi_1)\alpha_j(g_{ij}\circ\phi_1),\sum\limits_{i,j}(\psi_i\circ\phi_2)\alpha_j(g_{ij}\circ\phi_2))\\=\sum\limits_{j}\alpha_j((\psi\cdot g_j)_D\circ\phi)\in(\psi\cdot \cM)_D\circ\phi$.
\end{proof}

The next definition of Lipschitz saturation is motivated by Proposition \ref{proposition G} which relates $\overline{\cM}$ and $\overline{J_k(\cM)}$.

\begin{definition}
	Suppose that $\cM\subseteq\cO_X^p$ is an $\cO_{X}$-submodule of generic rank $k$ on each component of $X$.
	
	The \textbf{3-Lipschitz saturation of $\cM$ at $x\in X$} is denoted by $(\cM_{S_3})_x$, and is defined by $$(\cM_{S_3})_x:=\{h\in\cO_{X,x}^{p}\mid J_k(h,\cM)\subseteq(J_k(\cM))_S\mbox{ at }x\}.$$
	
	In the family case, the \textbf{3-Lipschitz saturation of $\cM$ at $x\in X$ relative to $Y$} is defined as above taking the Lipschitz saturation of $J_k(\cM)$ relative to $Y$.
\end{definition}

The next proposition allows us to see $\cM_{S_3}$ as a sheaf of $\cO_X$-submodules of $\cOXp$.

\begin{proposition} Suppose that $\cM\subseteq\cO_X^p$ is an $\cO_{X}$-submodule of generic rank $k$ on each component of $X$. Then:
	\begin{enumerate}
		\item $\cM_{S_3}$ is an $\cO_X$-submodule of $\cOXp$;
		\item $\cM\subseteq \cM_{S_3} \subseteq \overline{\cM}$. In particular, $\cM$ is a reduction of $\cM_{S_3}$ and $e(\cM,\cM_{S_3})=0$.
	\end{enumerate}
	The same result holds in the family case.
\end{proposition}

\begin{proof}
	(1) Let $g,h\in \cM_{S_3}$ at $x\in X$ and $\alpha\in\cO_{X,x}$. By the basic properties of determinants we have that $$J_k(\alpha g+h,\cM)\subseteq \alpha J_k(g,\cM)+J_k(h,\cM)\subseteq(J_k(\cM))_S.$$ 
	
	Hence, $\alpha g+h\in \cM_{S_3}$ at $x$.	 
	
	(2) Since $J_k(\cM)\subseteq (J_k(\cM))_S$ then $\cM\subseteq \cM_{S_3}$. Now, let $h\in \cM_{S_3}$ at $x$. Thus, $J_k(h,\cM)\subseteq (J_k(\cM))_S\subseteq\overline{J_k(\cM)}$ which implies that $h\in\overline{\cM}$.
\end{proof}

\begin{proposition}\label{P22}
	Suppose that $\cM$ has generic rank $k$ on each component of $X$. Then $\cM_{S_2}\subseteq\cM_{S_3}$.
\end{proposition}
\begin{proof}
	Suppose $h\in\cM_{S_2}$ at $x$. By Theorem 2.3 of \cite{G2} it is enough to prove that $(J_k(h,\cM))_D\subseteq\overline{(J_k(\cM))_D}$ at $(x,x)$. So, we have to prove that all the generators $(J_{IJ}(h,\cM))_D$ are in $\overline{(J_k(\cM))_D}$, for all $i\in\{1,...n\}$ and for all indexes $I$ and $J$. Write $I=(i_1,...,i_k)$ and $J=(j_1,...,j_k)$. We have two cases. 
	
	Suppose $j_1>1$. Then, as we saw before, $J_{IJ}(h,\cM)\in J_k(\cM)$, so $(J_{IJ}(h,\cM))_D\in (J_k(\cM))_D$, in particular, $(J_{IJ}(h,\cM))_D\in \overline{(J_k(\cM))_D}$.
	
	Suppose $j_1=1$. By Lemma \ref{L21} there exists $\psi:X\rightarrow \mbox{Hom}(\bC^p,\bC)$ such that $\psi\cdot h=J_{IJ}(h,\cM)$ and $\psi\cdot \cM\subseteq J_k(\cM)$. Since $h$ is in the 2-Lipschitz Saturation of $\cM$ at $x$ then $\psi\cdot h\in(\psi\cdot \cM)_S$, which is equivalent to $(\psi\cdot h)_D\in\overline{(\psi\cdot \cM)_D}$. Thus, $(J_{I,J}(h,\cM))_D=(\psi\cdot h)_D\in\overline{(\psi\cdot \cM)_D}\subseteq\overline{(J_k(\cM))_D}$.
\end{proof}

The next result shows the third saturation is closely related to the pullback of the saturated blow-up map.

\begin{proposition}\label{P20}
	Suppose that $\cM$ has generic rank $k$ on every component of $X$. Let $\pi: SB_{J_{k}(\cM)}(X)\rightarrow B_{J_{k}(\cM)}(X)$, $p: B_{J_{k}(\cM)}(X)\rightarrow X$ and $\pi_S=p\circ\pi$ be the projections maps. Let $h\in\cO_{X,x}^p$. Then: \begin{center} $h\in\cM_{S_3}$ if and only if $\pi_S^{*}(h)\in\pi_S^{*}(\cM)$.\end{center}
\end{proposition}

\begin{proof}
	Fix a set of generators $\{g_1,...,g_r\}$ of $\cM$.
	We work at $x'\in E$, $E=\pi^{-1}(E_B)$, $E_B$ the exceptional divisor of $B_{J_{k}(\cM)}(X)$. 
	
	Suppose first that $J_k(h,\cM)\subseteq (J_k(\cM))_S$. Let $\pi_S^{*}(J_{I,J}(\cM))$ be a local generator of the principal ideal $\pi_S^{*}(J_k(\cM))$. Then, by Cramer's rule we can write $$(J_{I,J}(\cM)\circ\pi_S)(h_I\circ\pi_S)=\sum\limits_{j\in J}(J_{I,j}(h_I,\cM_I)\circ\pi_S)(m_{I,j}\circ\pi_S) \eqno(1)$$ with $m_j\in \cM$. We claim in fact that $$(J_{I,J}(\cM)\circ\pi_S)(h\circ\pi_S)=\sum\limits_{j\in J}(J_{I,j}(h_I,\cM_I)\circ\pi_S)(m_{j}\circ\pi_S).$$ To see this pick a curve $\phi:(\bC,0)\rightarrow(SB_{J_k(\cM)}(X),x')$ and choose $\phi$ so that the rank of $\pi_S^*(\cM)|_{\phi}$ is generically $k$. Since by hypothesis $h\in\overline{\cM}$ then $h\circ\pi_S\circ\phi\in\phi^*(\pi_S^*(\cM))$. So, the element $$\star=h\circ\pi_S\circ\phi-\sum\limits_{j\in J}(\frac{J_{I,j}(h_I,\cM_I)\circ\pi_S}{J_{I,J}(\cM)\circ\pi_S}\circ\phi)(m_{j}\circ\pi_S\circ\phi)\in\phi^*(\pi_S^*(\cM))$$ because the above quotients by hypothesis are regular functions on $SB_{J_k(\cM)}(X)$. By equation (1), the above element has $0$ for the entries indexed by $I$. So, if $\star$ is not zero then $\phi^*(\pi_S^*(\cM))$ has rank at least $k+1$, which is a contradiction. Since the images of $\phi$ fill up a Z-open set it follows that $\star$ is locally zero. Therefore, $$h\circ\pi_S=\sum\limits_{j\in J}(\frac{J_{I,j}(h_I,\cM_I)\circ\pi_S}{J_{I,J}(\cM)\circ\pi_S})(m_{j}\circ\pi_S)\in\pi_S^*(\cM).$$
	
	Conversely, suppose that $\pi_S^{*}(h)\in\pi_S^{*}(\cM)$. Then we can write $$h\circ\pi_S=\sum\limits_{j=1}^{r}\alpha_j(g_j\circ\pi_S) \eqno(2)$$
	where $\alpha_j\in\widetilde{\cO}_{B_{J_k(\cM)}(X),x'}$, $\forall j\in\{1,...,r\}$. 
	
	\noindent Let $c$ be an arbitrary generator of $J_k(h,\cM)$. Then we can write $c=\det[h,\cM]_{IJ}$ for some $k$-indexes $I$ and $J$. If the $k$-index $J$ does not pick the first column of $[h,\cM]$, then $c$ is a $k\times k$ minor of the matrix $[\cM]$, hence $c\in J_k(\cM)\subseteq (J_k(\cM))_S$ and we are done. Thus we may assume that $j_1=1$. Then, we can write $$c=\det\left[\begin{matrix}
	h_I  &  (g_{j_1})_I  &... &(g_{j_{k-1}})_I
	\end{matrix}\right].$$
	Applying the pullback of the projection map we have $$\pi_S^*(c)=\det\left[\begin{matrix}
	h_I\circ\pi_S  &  (g_{j_1})_I\circ\pi_S  &... &(g_{j_{k-1}})_I\circ\pi_S
	\end{matrix}\right].$$
	
	\noindent By the equation (2), if we look only to the entries of the $k$-index $I$, we get $$h_I\circ\pi_S=\sum\limits_{j=1}^{r}\alpha_j((g_j)_I\circ\pi_S).$$
	
	\noindent Using the linearity of the determinant in the first column, we conclude that $\pi_S^*(c)=\sum\limits_{j=1}^{r}\alpha_j(v_j\circ\pi_S)$, where $v_j:=\det\left[\begin{matrix}
	(g_j)_I  &  (g_{j_1})_I  &... &(g_{j_{k-1}})_I
	\end{matrix}\right],$ for all $j\in\{1,..,r\}$. Clearly, $v_j\in J_k(\cM)$, $\forall j\in\{1,...,r\}$. Hence, $\pi_S^*(c)\in\pi_S^*(J_k(\cM))$ and by the definition of the Lipschitz saturation of an ideal, we conclude that $c\in(J_k(\cM))_S$. 
\end{proof}

\begin{corollary}\label{C21} Suppose that $\cM$ has generic rank $k$ on every component of $X$.
	If $h\in\cM_{S_3}$ then there exists an open covering of $X$ such that $h$ can be written on each $U$ of the covering as an element of $\cM$ by using Lipschitz functions (Lipschitz with respect to $(z_i)$ and $(\frac{J_{I',J'}(\cM)}{J_{I,J}(\cM)})$, where $J_{I,J}(\cM)$ is the minor which gives the local generator on the preimage of $U$ in $SB_{J_k(\cM)}(X)$) .
\end{corollary}

Now we look for conditions so that the above notions of Lipschitz saturation are equivalent.

The next proposition gives us such condition.

\begin{proposition}\label{T4.17}
Let $\cM$ be an $\cO_{X}$-submodule of $\cO_{X}^p$ with generic rank $k$ on each component of $X$, $x\in X$ and $h\in\cO_{X,x}^p$. Suppose that there exists an ideal $I$ of $\cO_{X\times X,(x,x)}^{2p}$ such that

$$ IJ_2((J_k(\cM))_D)\subseteq\overline{J_{2k}(\cM_D)}$$
$$J_{2k}(h_D,\cM_D)\subseteq\overline{IJ_2((J_k(h,\cM))_D)}$$ at $(x,x)$. Then:

 $$h\in(\cM_{S_1})_x \Leftrightarrow h\in(\cM_{S_2})_x \Leftrightarrow h\in(\cM_{S_3})_x $$
\end{proposition}

\begin{proof}
	The first implication follows from Proposition \ref{P23}. The second one follows from Proposition \ref{P22}.
	
	Suppose $h\in\cM_{S_3}$ at $x$. Then, $(J_k(h,\cM))_D\subseteq\overline{(J_k(\cM))_D}$ and so $\overline{(J_k(h,\cM))_D}=\overline{(J_k(\cM))_D}$.
	
	Let us prove that $J_{2k}(h_D,\cM_D)\subseteq\overline{J_{2k}(\cM_D)}$. In fact, let $\phi:(\bC,0)\rightarrow(X\times X,(x,x))$ be an arbitrary analytic curve. Since $\overline{(J_k(h,\cM))_D}=\overline{(J_k(\cM))_D}$ then $\phi^*((J_k(h,\cM))_D)=\phi^*((J_k(\cM))_D)$. Using the inclusions of the hypothesis and the curve criterion for the integral closure of modules we have $\phi^*(J_{2k}(h_D,\cM_D))\subseteq\phi^*(I\cdot J_2((J_k(h,\cM))_D))\\=\phi^*(I)\cdot J_2(\phi^*((J_k(\cM))_D))=\phi^*(I\cdot J_2((J_k(\cM))_D))\subseteq\phi^*(J_{2k}(\cM_D))$.
	
	By Theorem 2.9 of \cite{SG} we have that $\cM_D$ has generic rank $2k$ on each component of $X\times X$. Therefore, by Proposition \ref{proposition G} we conclude that $h_D\in\overline{\cM_D}$.  	 
\end{proof}	

In order to use the above criterion, we prove some useful lemmas that allow us to get the desired equivalence.

\textbf{Notation:} For each object $A$, we denote $A=A\circ\pi_1$ and $A'=A'\circ\pi_2$. For $k$-indexes $I,J$, we denote by $\cM_{IJ}$ the submatrix of $[\cM]$ defined by these $k$-indexes.

\begin{lemma}
	Let $\cM\subseteq\cO_{X,x}^p$ be a submodule and $k\in\bN$. Then $$(z_{t_1}-z_{t_1}')...(z_{t_k}-z_{t_k}')\det(\cM_{IJ})\det(\cM_{KL}')\in J_{2k}(\cM_D) \mbox{ at }(x,x)$$ for every $t_1,...,t_k\in\{1,...,n\}$ and $k$-indexes $I,J,K,L$.
	
\end{lemma}

\begin{proof}
	Write $\cM_{KL}=(g_{rs})$. The $2k\times 2k$ matrix $$\left[\begin{matrix}
	
	\cM_{IJ}   &  |  & 0_{k\times k} \\
	--------&      &  ------------\\
	 
	\cM_{IJ}'   &   |   &  
	\begin{matrix}
	(z_{t_1}-z_{t_1}')g'_{11}    &   ...   & (z_{t_k}-z_{t_k}')g'_{1k}\\
	       \vdots                &         &      \vdots    \\
	(z_{t_1}-z_{t_1}')g'_{k1}    &   ...   & (z_{t_k}-z_{t_k}')g'_{kk}       
	\end{matrix}

	\end{matrix}\right]$$ is a submatrix of $[\cM_D]$. 
	
	So, its determinant, which is $(z_{t_1}-z_{t_1}')...(z_{t_k}-z_{t_k}')\det(\cM_{IJ})\det(\cM_{KL}')$, belogns to $J_{2k}(\cM_D)$.	
\end{proof}

The next result gives us the first condition of Proposition \ref{T4.17} in terms of the ideal coming from the diagonal.

\begin{lemma}\label{L4.20}
	Let $\cM\subseteq\cO_{X,x}^p$ be a submodule and $k\in\bN$.
	\begin{itemize}
		\item[a)] $I_{\Delta}^{k}J_2((J_k(\cM))_D)\subseteq J_{2k}(\cM_D)$ at $(x,x)$;
		\item[b)] If $J_k(\cM)$ is principal then $I_{\Delta}^{k-1}J_2((J_k(\cM))_D)\subseteq J_{2k}(\cM_D)$ at $(x,x)$.
	\end{itemize}
\end{lemma}

\begin{proof}
	(a) We have that $$[(J_k(\cM))_D] =\left[\begin{matrix}
	
	\det(\cM_{IJ})    &  ...   &     0 \\
	\det(\cM_{IJ}')  & ... &   (z_i-z_i')\det(\cM_{KL}')\\ 	              
	\end{matrix}\right]$$
	varying the $k$-indexes $I,J,K,L$ and $i\in\{1,...,n\}$. Thus, the desired inclusion is a straightforward consequence of the previous lemma.
	
	(b) Since $J_k(\cM)$ is principal then there exist $k$-indexes $I,J$ such that $g=\det(\cM_{IJ})$ and $J_k(\cM)$ is generated by $\{g\}$.Thus we can write $$[(J_k(\cM))_D] =\left[\begin{matrix}
	
	g    &      0 \\
	g'   &   (z_i-z_i')g'\\ 	              
	\end{matrix}\right]$$
	varying $i\in\{1,...n\}$. So, in this case $J_2((J_k(\cM))_D)$ is generated by $\{g.g'(z_i-z_i')\mid i\in\{1,...,n\}\}$. By previous lemma we have that $$[(z_{t_1}-z_{t_1}')...(z_{t_{k-1}}-z_{t_{k-1}}')].[g.g'(z_i-z_i')]\in J_{2k}(\cM_D),$$
	
	for all $t_1,...,t_{k-1},i\in\{1,...,n\}$.
	
	Therefore, $I_{\Delta}^{k-1}J_2((J_k(\cM))_D)\subseteq J_{2k}(\cM_D)$ at $(x,x)$. 
\end{proof}

Now, we start to get conditions in order to obtain the second condition of Proposition \ref{T4.17}.

\begin{lemma}
	Let $\cM$ be an $\cO_{X,x}$-submodule of $\cO_{X,x}^p$ and $k\in\bN$. Denote by $[\tilde{\cM}]$ the matrix such that $$[\cM_D]=\left[\begin{matrix}
	[\cM]        &     0 \\
	[\cM]'        &    [\tilde{\cM}]
	\end{matrix}\right].$$
	
	Let $\cJ_{2k}(\cM_D)$ be the subideal of $J_{2k}(\cM_D)$ generated by $$\{\det(\cM_{IJ})\det(\tilde{\cM}_{KL})\mid I,J,K,L\mbox { are }k\mbox{-indexes}\}.$$
	
	Then, $$\cJ_{2k}(\cM_D)\subseteq I_{\Delta}^{k-1}J_2((J_k(\cM))_D)\mbox { at }(x,x) .$$
\end{lemma}

\begin{proof}
	It suffices to prove that each generator of $\cJ_{2k}(\cM_D)$ belongs to $I_{\Delta}^{k-1}J_2((J_k(\cM))_D)$.
	
	The columns of $\tilde{\cM}_{KL}$ are columns of $[\cM]'$ possibly in a different order with terms of type $z_i-z_i'$ multiplied on each column. If the $k$-indexes $K,L$ pick repeated columns then $\det(\tilde{\cM}_{KL})=0$ and we are done. If this does not occur, then 
	$$\det(\tilde{\cM}_{KL})=(z_{i_1}-z_{i_1}')...(z_{i_k}-z_{i_k}')(\pm \det(\cM'_{K'L'}))$$ for some reorganization $k$-indexes $K',L'$, where $i_1,...,i_k\in\{1,...,n\}$ are the indexes which $z_{i_1}-z_{i_1}',...,z_{i_k}-z_{i_k}'$ appear on each of the $k$ columns of $\tilde{\cM}_{KL}$. Since $\pm(z_{i_1}-z_{i_1}')...(z_{i_{k-1}}-z_{i_{k-1}}')\in I_{\Delta}^{k-1}$ and \\$\det(\cM_{IJ})(z_{i_k}-z_{i_k}')\det(\cM'_{KL})\in J_2((J_k(\cM))_D)$ then $$\det(\cM_{IJ})\det(\tilde{\cM}_{KL})\in I_{\Delta}^{k-1}J_2((J_k(\cM))_D).$$
\end{proof}

The next result ensures that any free submodule satisfies the second condition of Proposition \ref{T4.17} for a suitable power of the ideal coming from the diagonal.

\begin{lemma}
	Let $\cM$ be a free $\cO_{X,x}$-submodule of $\cO_{X,x}^p$ of rank $k$. Then $$J_{2k}(\cM_D)\subseteq I_{\Delta}^{k-1}J_2((J_k(\cM))_D)\mbox{ at }(x,x).$$
\end{lemma}

\begin{proof}
 	We have that $$[\cM_D]=\left[
 	\begin{matrix}
 	[\cM]_{p\times k}   &    |    &     0   \\
 	    ------        &    |    &  -------  \\
 	[\cM]'_{p\times k} & |    &  (z_i-z_i')[\cM]'
 	
 	\end{matrix}
 	\right] $$
 	varying $i\in\{1,...,n\}$.
 	
 	Let $d\in J_{2k}(\cM_D)$ at $(x,x)$ be an arbitrary generator. Then $d=\det N$ where $N$ is a $2k\times 2k$ submatrix of $[\cM_D]$.
 	
 	(i) Suppose first that there are $k+t$ columns of $N$ taken on the part $$\left[\begin{matrix}
 	0 \\
 	(z_i-z_i')[\cM]'
 	\end{matrix}\right]$$
 	of $[\cM_D]$, with $1\leq t \leq k$. Then we can write $$N=\left[\begin{matrix}
 	(\cM_{IJ})_{k\times(k-t)}    &    |    &   0_{k\times({k+t})}\\
 	--------         &    |    &   --------\\
 	(\cM_{IJ})'_{k\times(k-t)}   &    |    &  (\tilde{\cM})_{KL}
 	\end{matrix}\right]$$
 	
 	$$=\left[\begin{matrix}
 	\cM_{IJ}   &    0_{k\times t}  &   |  &   0_{k\times k}\\
 	---    & ---        &   |  &  ------\\
 	\cM_{IJ}'  &     *             &   |  &    **
 	\end{matrix}\right]$$ for some matrices * and **, where ** is a square matrix of size $k\times k$. So in this case we have that $d=\det N=\det\left[\begin{matrix}
 	\cM_{IJ}  &   0_{k\times t}
 	\end{matrix}\right].\det(**)=0\in I_{\Delta}^{k-1}J_2((J_k(\cM))_D)$.  	
 	
 	(ii) Now suppose that we have exactly $k$ columns of $N$ taken on the part $$\left[\begin{matrix}
 	0 \\
 	(z_i-z_i')[\cM]'
 	\end{matrix}\right]$$
 	of $[\cM_D]$. Then we can write $$N=\left[\begin{matrix}
 	\cM_{IJ}    &    |      &     0_{k\times k}  \\
 	------    &    |      &     ------ \\
 	\cM_{IJ}'   &    |      &     \tilde{\cM}_{KL}
 	\end{matrix}\right].$$
 	
 Thus, $d=\det(\cM_{IJ})\det(\tilde{\cM}_{KL})\in\cJ_{2k}(\cM_D)$ and by previous lemma we conclude that $d\in I_{\Delta}^{k-1}J_2((J_k(\cM))_D)$.
\end{proof}

As a consequence of the previous lemma, we have a result which states the second condition of Proposition \ref{T4.17} holds locally in a dense Zariski open subset of $X$ in a sheaf of submodules of $\cOXp$.

\begin{lemma}\label{L4.23}
	Let $\cM$ be an $\cO_X$-submodule of $\cOXp$ of generic rank $k$ on each component of $X$. Then there exists a dense Zariski open subset $U$ of $X$ such that $U\cap V$ is a dense Zariski open subset of $V$, $\forall V$ component of $X$, $J_k(\cM)=\cO_{X,x}$ and $$J_{2k}(\cM_D)\subseteq I_{\Delta}^{k-1}J_2((J_k(\cM))_D)\mbox{ at }(x,x),$$ for all $x\in U$.
\end{lemma}

\begin{proof}
	Consider $[\cM]$ a matrix of generators of $\cM$. Using the Cramer's rule, we can choose $k$ $\cO_X$-linear independents columns of $[\cM]$ such that these columns generates $\cM$ in a such dense Zariski open subset $U$ of $X$ and $J_k(\cM)=\cO_X$ along $U$. Let $\cM_k$ be the $\cO_X$-submodule of $\cOXp$ generated by the columns chosen above. Thus, given $x\in U$, we have that $\cM_x=(\cM_k)_x$ is a free $\cO_{X,x}$-submodule of $\cO_{X,x}^p$ of rank $k$ and the desired inclusion is a consequence of the previous lemma.  
\end{proof}

Before we state the main theorem of this section, we prove that the generic rank of $(h,\cM)$ is the same as the generic rank of $\cM$, provided $h\in\overline{\cM}$.

\begin{lemma}
	Let $\cM$ be an $\cO_{X,x}$-submodule of $\cO_{X,x}^p$ of generic rank $k$ on each component of $X$. If $h\in\overline{\cM}$ then $(h,\cM)$ also has generic rank $k$ on each component of $X$.
\end{lemma}

\begin{proof}
Since $h\in\overline{\cM_x}$ then $h\in\overline{\cM_{x'}}$, for all $x'$ in an open neighborhood $U$ of $x$ in $X$. Suppose the rank of $(h,\cM)$ is $k+1$. Then $h(x')\notin \cM(x')$. But, $x'$ is a constant curve, so $h\notin\overline{\cM_{x'}}$, contradiction.
\end{proof}

The next theorem gives us the desired equivalence.

\begin{theorem}
	Let $\cM$ be an $\cO_X$-submodule of $\cO_X^p$ of generic rank $k$ on each component of $X$. Then there exists a dense Zariski open subset $U$ of $X$ 
	 such that	$$ \cM_{S_1}=\cM_{S_2}=\cM_{S_3}$$ along $U$. 
\end{theorem}

\begin{proof}
	Since $X\subseteq\bC^n$ is an analytic complex variety then $\cO_X$ is a noetherian sheaf of rings, hence $\cO_X^p$ is a noetherian sheaf of $\cO_X$-modules. Since $\cM_{S_3}$ is a sheaf of $\cO_X$-modules then $\cM_{S_3}$ is finitely generated by global sections $h_1,...,h_r$. Since $h_i\in\overline{\cM}$, $\forall i\in\{1,...,r\}$ then by previous lemma $(h_i,\cM)$ also has generic rank $k$ on each component of $X$, $\forall i\in\{1,...,r\}$. By Lemma \ref{L4.23}, for each $i\in\{1,...,r\}$ there exists a dense Zariski open subset $U_i$ of $X$ such that $U_i\cap V$ is a dense Zariski open subset of $V$, for all $V$ component of $X$, and $$J_{2k}((h_i)_D,\cM_D)\subseteq I_{\Delta}^{k-1}J_2((J_k(h_i,\cM))_D)\mbox{ at }(x,x),$$ $\forall x\in U_i$.
	
	Also by Lemma \ref{L4.23} there exists a dense Zariski open subset $U_0$ of $X$ such that $U_0\cap V$ is a dense Zariski open subset of $V$, for all component $V$ of $X$, and $J_k(\cM)=\cO_{X,x}$ which is principal at $x$, $\forall x\in U_0$. From Lemma \ref{L4.20}(b), we conclude that $$I_{\Delta}^{k-1}J_2((J_k(\cM))_D)\subseteq J_{2k}(\cM_D)\mbox{ at }(x,x),$$
	$\forall x\in U_0$. 
	
	Take $U:=\bigcap_{j=0}^r U_j$. Then $U$ is a dense Zariski open subset of $X$ and $U\cap V$ is a dense Zariski open subset of $V$, for all $V$ component of $X$.
	
	We already know that $\cM_{S_1}\subseteq \cM_{S_2}\subseteq \cM_{S_3}$ at any point of $X$, in particular, along $U$. Let us prove that $\cM_{S_3}\subseteq \cM_{S_1}$ along $U$. In fact, let $x\in U$. Given an arbitrary $i\in\{1,...,r\}$, since $x\in U_0$ and $x\in U_i$ then: 
	\begin{itemize}
		\item $I_{\Delta}^{k-1}J_2((J_k(\cM))_D)\subseteq J_{2k}(\cM_D)\mbox{ at }(x,x)$;
		
		\item 	$J_{2k}((h_i)_D,\cM_D)\subseteq I_{\Delta}^{k-1}J_2((J_k(h_i,\cM))_D)\mbox{ at }(x,x)$. 
	\end{itemize} 
	
	Since $h_i\in \cM_{S_3}$ at $x$ then by Proposition \ref{T4.17} we have that $h_i \in \cM_{S_1}$ at $x$, $\forall i\in\{1,...,r\}$. Since $\cM_{S_3}$ at $x$ is generated by $h_1,...,h_r$ at $x$, then $\cM_{S_3}\subseteq \cM_{S_1}$ at $x$.
\end{proof}

In the next example we show that in general $\cM_{S_3}$ is not a submodule of $\cM_{S_1}$.

\begin{example}
	Consider the submodule $\cM$ of $\cO_2^2$ given by \\$[\cM]=\left[\begin{matrix}
	x  &  0  &   y\\
	y  &  x  &   0
	\end{matrix}\right]$ and let $h=(x,3y)$. 
	
	Thus, $J_2(\cM)=\langle  x^2,xy,y^2\rangle=J_2(h,\cM)$. In particular, $h\in\cM_{S_3}$.
	
	Consider the curve $\Phi:(\bC,0)\rightarrow(\bC^2\times\bC^2,(0,0))$ given by $\Phi(t)=(t,\alpha t, t, \beta t)$.
	
	We have that $$[\Phi^*(\cM_D)]=\left[\begin{matrix}
	t & 0 & \alpha t & 0 & 0 & 0 \\
	\alpha t & t & 0 &  0 & 0 & 0\\
	t & 0 & \beta t & (\alpha-\beta)t^2 & 0 & \beta(\alpha-\beta)t^2\\
	\beta t & t & 0 & \beta(\alpha-\beta)t^2 & (\alpha-\beta)t^2 & 0
	\end{matrix}\right]$$
	
	\noindent and $[\Phi^*(h_D)]=\left[\begin{matrix}
	t\\
	3\alpha t \\
	t\\
	3\beta t
	\end{matrix}\right]$. Let $D$ be the first $4\times 3$ submatrix of $[\Phi^*(M_D)]$. The complementar submatrix has degree 2, so we can ignore its terms. Notice that:
	
	$[\Phi^*(h_D)]-D\cdot e_1=\left[\begin{matrix}
	0\\
	2\alpha t \\
	0\\
	2\beta t
	\end{matrix}\right]$ and $[\Phi^*(h_D)]-D\cdot (2\alpha e_2)=\left[\begin{matrix}
	0\\
	0 \\
	0\\
	(2\beta-2\alpha) t
	\end{matrix}\right]$. However, the first $3$ vectors of $D$ generically have rank $3$. Therefore, $h\notin \cM_{S_1}$.
\end{example}

\section{The Necessity Theorem}\label{sec4}

In this section we prove the Necessity theorem, and we relate the infinitesimal Lipschitz condition coming from the first Lipschitz saturation with the strongly bi-Lipschitz equisingularity and the Mostowski's conditions for a Lipschitz stratification.

Let $X\subseteq\bC^n$ be an analytic variety. Let $\cL_X$ be the sheaf of locally Lipschitz functions on $X$. Clearly $\cO_X$ is a subsheaf of $\cL_X$. 

\begin{definition}
	Let $\cM\subseteq\cOXp$ be an $\cO_X$-submodule. We define $\cM_{\cL}$ as the $\cO_X$-submodule of $\cO_X^p$ given by $\cM_{\cL}:=\cL_X\cM\cap\cOXp$. The module $\cM_{\cL}$ is called \textbf{strong Lipschitz saturation of $\mathbf{\cM}$}.
\end{definition}

\begin{theorem}[Necessity Theorem]\label{T2.1}
	Let $\cM\subseteq\cO_{X}^p$ be an $\cO_X$-submodule of $\cOXp$ generated by global sections $\{h_1,...,h_r\}$, $x\in X$. If $h\in \cM_{\cL}$ at $x$ then $h_D\in \overline{\cM_D}$ at $(x,x')$, for any $x'\in X$. In particular, $\cM_{\cL}\subseteq\cM_{S_1}$.
\end{theorem}
\begin{proof}
	The generators of generators of $\cM_D$ are given by the set $$B=\{(h_i)_D\}_{i=1}^r\cup\{(0,(z_j\circ\pi_1-z_j\circ\pi_2)(h_i\circ\pi_2))\}_{i,j=1}^{r,n} .$$	
	
	We will use the criterion for integral closure of modules obtained in Proposition 1.11 of \cite{G3}. For each $i\in\{1,...,r\}$ we can write \\$h_{i}=(h_{i1},...,h_{ip})$. 
	
	Let $\varphi:X\times X\rightarrow \mbox{Hom}(\bC^{2p},\bC)$ be an arbitrary global section of $\Gamma(\mbox{Hom}(\bC^{2p},\bC),X\times X)$ with components $\varphi_1,...,\varphi_{2p}$. 
	
	For each $(z,z')\in X\times X$ let $$M(z,z'):=\sup\{\norm \varphi(z,z')\cdot s(z,z')\norm\mid s\in B \}.$$
	
	Since $h\in\cOXp$ then $h_D\in\cO_{X\times X}^{2p}$. By hypothesis, for each $i\in\{1,...,r\}$ there exists $C_i>0$ such that $$\norm \alpha_i(z)-\alpha_i(z')\norm \leq C_i\underset{j\in\{1,...,n\}}{\sup}\{\norm z_j-z'_j\norm\}$$
	\noindent $\forall (z,z')\in X\times X$. 
	
	Take $C_0:=\max\{C_1,...,C_r\}$. For each $i\in\{1,...,r\}$, $\alpha_i$ is locally bounded, so there exists $D_i>0$ and $U_i$ neighborhood of $x$ in $X$ such that $\norm \alpha_i(z)\norm\leq D_i$, $\forall z\in U_i$. Take $U:=\bigcap\limits_{i=1}^r U_i$ and $D:=\max\{D_1,...,D_r\}$.
	
	Let $(z,z')\in U\times X$ be arbitrary. Then: $\norm \varphi(z,z')\cdot h_D(z,z')\norm 
	\leq \sum\limits_{i=1}^{r}\norm \alpha_i(z)\norm \norm \varphi(z,z')\cdot (h_i)_D(z,z')\norm + \sum\limits_{i=1}^{r}\norm \alpha_i(z)-\alpha_i(z')\norm \norm \sum\limits_{\ell=1}^{p}\varphi_{p+\ell}(z,z')h_{i\ell}(z,z')\norm \\
	\leq DrM(z,z')+C_0r\underset{\begin{smallmatrix}
		i\in\{1,...,r\}\\
		j\in\{1,...,n\}
		\end{smallmatrix}}{\sup}\{\norm \sum\limits_{\ell=1}^{p}\varphi_{p+\ell}(z,z')(z_j-z'_j)h_{i\ell}(z')\norm\}
	   \\ \leq(Dr+C_0r)M(z,z')$. Taking $C:=C_0r+Dr>0$ we just proved that $$\norm \varphi(z,z')\cdot h_D(z,z')\norm\leq C\underset{s\in B}{\sup}\{\norm \varphi(z,z')\cdot s(z,z')\norm \} $$
	\noindent for all $(z,z')\in U\times X$, where $B$ is a generator set of $M_D$ and $U\times X$ is a neighborhood of $(x,x')$ in $X\times X$. Therefore, $h_D\in\overline{M_D}$ at $(x,x')$.
\end{proof}

\begin{remark}
	If the elements of $\cL_X$ were meromorphic functions the above theorem would follow easily. However we assume nothing about them other than that they are Lipschitz.
\end{remark}

\begin{remark}
	The Necessity Theorem still works for the real analytic case. One can use the Proposition 4.2 of \cite{G3} and the proof is the same.
\end{remark}

\begin{corollary} Let $\cM\subseteq\cO_{X}^p$ be an $\cO_X$-submodule of $\cOXp$. Then, $\cM\subseteq \cM_{\cL}\subseteq \overline{\cM}$ and, in particular, 	$e(\cM,\cM_{\cL})=0$.
\end{corollary}

In \cite{FR2} the authors defined the notion of \textit{strongly bi-Lipschitz equisingularity} for a family of analytic complex functions.

\begin{definition}
	Let $U$ be a connected open subset of $\bC$ and let $(X,0)\subseteq(\bC^n,0)$ be a germ of an analytic variety. Consider a family of analytic functions $F:X\times U\rightarrow \bC$ with $F(0,y)=0$ and $f_y(z)=F(z,y)$. We say that $F$ is \textbf{strongly bi-Lipschitz equisingular (sLe)}  if there exists a continuous map $v:X\times\bC\rightarrow\bC^n$, $v_y(z):=v(z,y)$, such that each vector field $v_y:\bC^n\rightarrow \bC^n$ is Lipschitz, $\forall y\in U$ and $$\frac{\partial f_y}{\partial y}(z)=D(f_y)_z(v_y(z))$$
	\noindent for every $y\in U$ and $z$ near to the origin in $X$. Notice the last equation can be rewritten as $$\frac{\partial F}{\partial y}=\sum\limits_{i=1}^{n}v_i\frac{\partial F}{\partial z_i}$$
	\noindent where $v_1,...,v_n$ are the components of $v$. 
\end{definition}

So the sLe condition is equivalent to asking that the partial derivative of $F$ is an element of the strong Lipschitz saturation of the jacobian ideal $J(F)$.

In \cite{G4} the author defined a condition for a family of functions analogous to the $iL_A$ condition for sets (see \cite{SG,G1}).

\begin{definition}
	Let  $F:\bC^n\times \bC\rightarrow \bC$ be a family of analytic fucntions with $F(0,y)=0$ and $f_y(z)=F(z,y)$. Denote the parameter space $Y=0\times \bC$ and assume that $Y$ is the singular locus of $F$. Here we work with the double relative to $Y$, i.e, defined by the projections $\pi_1,\pi_2$ restricted to the fiber product $(\bC^n\times\bC)\underset{Y}{\times}(\bC^n\times\bC)$. We say that $F$ is \textbf{infinitesimally Lipschitz equisingular (iLe)} at a point of the fiber product if $$(JM(F)_Y)_D\subseteq\overline{(JM_z(F))_D}$$ at this point.
\end{definition}

Using the Necessity Theorem, we get the following

\begin{corollary}\label{T2.5}
	If $F:\bC^n\times\bC\rightarrow \bC$ is a strongly bi-Lipschitz equisingular family of functions then $F$ is infinitesimally Lipschitz equisingular.
\end{corollary}

\begin{corollary}
	Let $X\subseteq\bC^n$ be an analytic variety. If $S$ is a stratum of the Mostowski stratification of $X$ then the $iL_A$ condition holds along $S$.
\end{corollary}

\begin{example}[Fernandes and Ruas example]
	Consider the family of functions given by	
	$F(x,y,t)=\frac{1}{3}x^3-t^2xy^{3n-2}+y^{3n}$. In \cite{FR} the authors proved this family is not strongly bi-Lipschitz trivial. Here we use our criterion to get the same conclusion. In fact, we have:
	
	$\frac{\partial F}{\partial t}=-2txy^{3n-2}$
	
	$\frac{\partial F}{\partial x}=x^2-t^2y^{3n-2}$
	
	$\frac{\partial F}{\partial y}=(2-3n)t^2xy^{3n-3}+3ny^{3n-1}$.
	
	Let us prove that $iL_A$ does not hold for the family $F$. Consider the curve $\phi(t)=(ty^{\frac{3n-2}{2}},y,t,-ty^{\frac{3n-2}{2}},y,t)$, $y=t^2$, which is formed by the two parametrizations $\phi_1$ and $\phi_2$ of the two branches of the polar curve defined by $\frac{\partial F}{\partial x}=0$.
	
	Since $(\frac{\partial F}{\partial x})_D\circ\phi=(0,0)$ and $y\circ\phi_1-y\circ\phi_2=0$ then $\phi^*((J_z(F))_D)$ has two generators $(\frac{\partial F}{\partial y}\circ\phi_1,\frac{\partial F}{\partial y}\circ\phi_2)$ and $(0,(x\circ\phi_1-x\circ\phi_2)(\frac{\partial F}{\partial y}\circ\phi_2))$. Thus, if $F$ satisfies $iL_A$ then we nust have \\ $\frac{\partial F}{\partial t}\circ\phi_2-\left((\frac{\partial F}{\partial t}\circ\phi_1)\left(\frac{\frac{\partial F}{\partial y}\circ\phi_2}{\frac{\partial F}{\partial y}\circ\phi_1}\right)\right)\in((\frac{\partial F}{\partial y}\circ\phi_2)(x\circ\phi_1-x\circ\phi_2))$. This implies that $$2t^2y^{\frac{9n-6}{2}}\left(1+\frac{3n+(3n-2)t^3y^{\frac{3n-6}{2}}}{3n+(2-3n)t^3y^{\frac{3n-6}{2}}}\right)\in(ty^{\frac{9n-4}{2}})=(ty^{\frac{9n-6}{2}}.y)$$
	
	\noindent which is a contradiction. Hence, $F$ does not satisfy $iL_A$ condition. Hence, $F$ is not strongly bi-Lipschitz trivial.
\end{example}

\begin{remark}
	Since Mostowski showed that any complex analytic has a Mostowski stratification, when coupled with our necessity result, this gives another proof of the genericity of the $iL_A$ condition.
\end{remark}


\vspace{1cm}

{\sc Terence James Gaffney
	
	\vspace{0.5cm}
	
	{\small Department of Mathematics 
		
		Northeastern University
		
		{\tiny 567 Lake Hall - 02115 - Boston - MA, USA, t.gaffney@northeastern.edu }}}

\vspace{1cm}	

{\sc Thiago Filipe da Silva
	
	\vspace{0.5cm}
	
	{\small Department of Mathematics
		
		Federal University of Esp\'irito Santo
		
		{\tiny Av. Fernando Ferrari, 514 - Goiabeiras, 29075-910 - Vit\'oria - ES, Brazil, thiago.silva@ufes.br}}}

\end{document}